\documentclass[11pt,twoside]{article}
\usepackage[margin=1.3in]{geometry}

\usepackage{latexsym,amsfonts,amsmath,amsthm,amssymb,makeidx}
\usepackage[title]{appendix}
\usepackage{CJK,CJKnumb,CJKulem,times,dsfont,ifthen,mathrsfs,latexsym,amsfonts, color}
\usepackage{amsmath,amsthm,makeidx,fontenc,amssymb,bm,graphicx,psfrag,listings, curves,extarrows}
\usepackage{cite}
\usepackage[colorlinks,citecolor=black,linkcolor=black]{hyperref}

\let\oldbibliography\thebibliography
\renewcommand{\thebibliography}[1]{%
\oldbibliography{#1}%
\setlength{\itemsep}{0pt}%
}

\makeindex
\newtheorem{theorem}{Theorem}[section]
\newtheorem{lemma}[theorem]{Lemma} 

\newtheorem{definition}[theorem]{Definition}
\newtheorem{example}[theorem]{Example}
\newtheorem{corollary}[theorem]{Corollary}
\newtheorem{remark}[theorem]{Remark}

\newcommand{\R}{\mathbb R}

\newcommand{\bt}{\begin{theorem}}
\newcommand{\et}{\end{theorem}}
\newcommand{\bl}{\begin{lemma}}
\newcommand{\el}{\end{lemma}}
\newcommand{\bd}{\begin{definition}}
\newcommand{\ed}{\end{definition}}
\newcommand{\bc}{\begin{corollary}}
\newcommand{\ec}{\end{corollary}}
\newcommand{\bp}{\begin{proof}}
\newcommand{\ep}{\end{proof}}
\newcommand{\bx}{\begin{example}}
\newcommand{\ex}{\end{example}}
\newcommand{\bi}{\begin{exercise}}
\newcommand{\ei}{\end{exercise}}
\newcommand{\bo}{\begin{prop}}
\newcommand{\eo}{\end{prop}}
\newcommand{\br}{\begin{remark}}
\newcommand{\er}{\end{remark}}
\newcommand{\be}{\begin{equation}}
\newcommand{\ee}{\end{equation}}
\newcommand{\ba}{\begin{align}}
\newcommand{\ea}{\end{align}}
\newcommand{\bn}{\begin{enumerate}}
\newcommand{\en}{\end{enumerate}}
\newcommand{\bg}{\begin{align*}}
\newcommand{\bcs}{\begin{cases}}
\newcommand{\ecs}{\end{cases}}

\newcommand{\bean}{\begin{eqnarray*}}
\newcommand{\eean}{\end{eqnarray*}}

\begin{document}
\title{ {\bf Large  singular  solutions  for  conformal  $Q$-curvature  equations  on  $\mathbb{S}^n$} }
\date{September   25,     2020}   
\author{\\  Xusheng Du,   \;\;        Hui Yang } 
\maketitle

\begin{center}
\begin{minipage}{130mm}
\begin{center}{\bf Abstract}\end{center}
In this paper, we study the existence of positive functions $K \in C^1(\mathbb{S}^n)$ such that the conformal $Q$-curvature equation  
\begin{equation}\label{001}
P_m (v) =K v^{\frac{n+2m}{n-2m}}~~~~~~ \textmd{on} ~ \mathbb{S}^n
\end{equation}
has a singular positive solution $v$  whose singular set is  a single point,  where $m$ is an integer satisfying $1 \leq m  < n/2$  and  $P_m$ is the intertwining operator of order $2m$. More specifically, we show that when $n\geq 2m+4$, every positive function in  
$C^1(\mathbb{S}^n)$ can be approximated in the $C^1(\mathbb{S}^n)$ norm by a positive function $K\in C^1(\mathbb{S}^n)$ such that \eqref{001} has a singular positive solution whose singular set is a single point. 
Moreover, such a solution can be constructed to be arbitrarily large near its singularity.  This is in contrast to the well-known results of Lin  \cite{Lin1998} and Wei-Xu \cite{Wei1999} which show that Eq. \eqref{001}, with $K$ identically a positive constant on $\mathbb{S}^n$, $n > 2m$, does not exist a singular positive solution whose singular set is a  single point.

\vskip0.10in


\vskip0.10in


\end{minipage}

\end{center}

\vskip0.20in

\section{Introduction}
Let  $\mathbb{S}^n$ be the $n$-dimensional unit sphere endowed with the induced metric $g_{\mathbb{S}^n}$ from $\mathbb{R}^{n+1}$.    The  aim of  this  paper is to  study  the  existence of positive functions $K\in C^1(\mathbb{S}^n)$ such that the conformal $Q$-curvature equation 
\begin{equation}\label{S001}
P_m (v) =K v^{\frac{n+2m}{n-2m}}~~~~~~ \textmd{on} ~ \mathbb{S}^n
\end{equation}
 has a singular positive solution $v$  whose singular set consists of a single point,    where $m$ is an integer,  $1 \leq m  < n/2$,   and $P_m$ is  an intertwining operator (see, e.g., Branson \cite{B95}) of order $2m$ given by 
\begin{equation}\label{Inter1}
P_m=\frac{\Gamma(B + \frac{1}{2} + m)}{\Gamma(B + \frac{1}{2} - m)},~~~~ B =\sqrt{-\Delta_{g_{\mathbb{S}^n}} + \left( \frac{n-1}{2} \right)^2}
\end{equation}
with $\Gamma$ being  the  Gamma function and  $\Delta_{g_{\mathbb{S}^n}}$ being  the Laplace-Beltrami operator on $(\mathbb{S}^n,  g_{\mathbb{S}^n})$.  The operator $P_m$ can be seen more concretely on $\mathbb{R}^n$ via the stereographic projection. Let $N$ be the north pole of $\mathbb{S}^n$  and $F$ be the inverse of the  stereographic projection:
$$
F:  \mathbb{R}^n \rightarrow  \mathbb{S}^n \backslash \{N\}, ~~~~  x \mapsto \left( \frac{2x}{1 + |x|^2},  \frac{|x|^2 - 1}{ |x|^2 +1} \right).  
$$ 
Then it follows from the conformal invariance of $P_m$ that 
\begin{equation}\label{Inter2}
(P_m(\phi))  \circ  F = |J_F|^{-\frac{n+2m}{2n}} (-\Delta)^m ( |J_F|^{\frac{n-2m}{2n}} (\phi \circ F))~~~~ \textmd{for}  ~ \phi \in C^\infty(\mathbb{S}^n), 
\end{equation}
where $\Delta$  is the Laplacian  operator on  $\mathbb{R}^n$   and $|J_F|$ is the determinant of the Jacobian of $F$, i.e.,  
$$
|J_F| = \left( \frac{2}{1 + |x|^2} \right)^n.   
$$

When $m=1$, $P_1= -\Delta_{g_{\mathbb{S}^n}} + \frac{n-2}{4(n-1)} R_0 $ is the well-know {\it conformal  Laplacian} associated with the metric $g_{\mathbb{S}^n}$, where $R_0=n(n-1)$ is the scalar curvature of $(\mathbb{S}^n,  g_{\mathbb{S}^n})$.  In this case, the equation \eqref{S001} reads as 
\begin{equation}\label{Sca01}
 -\Delta_{g_{\mathbb{S}^n}} v + \frac{n(n-2)}{4} v =K v^{\frac{n+2}{n-2}}~~~~~~ \textmd{on} ~ \mathbb{S}^n, ~~~ n \geq 3, 
\end{equation}
which is usually called the conformal scalar curvature equation.   Equation \eqref{Sca01} naturally arises in the study of the classical Nirenberg  problem that asks:   Which function $K$ on $(\mathbb{S}^n, g_{\mathbb{S}^n})$ is the scalar curvature of a metric $g$ that is conformal to $g_{\mathbb{S}^n}$?  There have been many papers on this problem,  see,  for example,  \cite{BC,CY,Chang1993,Chen1997A,CX12,Li01,Li02,SZ96,WY}  and the references therein.  Furthermore, the classical works of Schoen and Yau \cite{S,SY} on conformally flat manifolds have indicated the importance of studying  singular positive  solutions of \eqref{Sca01}.    When $K$ is identically a positive constant on $\mathbb{S}^n$,   Caffarelli, Gidas and Spruck \cite{Caffarelli1989} proved  that Eq. \eqref{Sca01} does not have a singular positive solution whose singular set consists of a single point, while Schoen \cite{S} constructed a singular positive solution of \eqref{Sca01} whose  singular set is any prescribed finite collection of at least {\it two}  points.  Mazzeo and Pacard in \cite{MP} provided  another construction method for Schoen's result.   When $K$ is a non-constant positive function,  Taliaferro \cite{Taliaferro2005} studied the existence of $K\in C^1(\mathbb{S}^n)$ such that \eqref{Sca01} has a singular positive solution $v$ whose singular set is  a single point  in  dimension $n \geq 6$,   and the solution $v$ can be constructed to be arbitrarily large near its singularity. This shows  that  there  does not exist an a priori estimate on the blow up rate of such a solution  $v$  near its singularity.    On the other hand,   Chen and Lin in  a series of  papers \cite{Chen1997E,Chen1999,Lin2000}  studied, among other things,  that under what assumptions on $K$  a  singular positive solution $v$ of \eqref{Sca01} near its singularity $\xi_0$ satisfies the following a priori estimate
$$
v(\xi) =O(|\xi - \xi_0|^{-\frac{n-2}{2}}). 
$$
We may also  see  \cite{TaliaferroZhang2006,Z02} for the similar a priori estimates of  singular positive solutions of \eqref{Sca01}.

When $m=2$, $P_2$ is the fourth order conformally invariant {\it Paneitz operator} associated with the metric $g_{\mathbb{S}^n}$.   For a smooth compact $n$-dimensional Riemannian manifold $(M, g)$ with $n\geq 4$,  the Branson's $Q$-curvature (see Branson \cite{B85}) is given by 
$$
Q_g=-\frac{1}{2(n-1)} \Delta_g R_g + \frac{n^3 - 4n^2 +16n -16}{8(n-1)^2(n-2)^2}  R_g^2 -\frac{2}{(n-2)^2} |Ric_g|^2, 
$$
where $R_g$ and $Ric_g$ denote the scalar curvature and the Ricci tensor of $g$ respectively.  The  fourth order  Paneitz operator (see  \cite{B85,P}) with respect to the metric $g$ is  defined by 
$$
P_2^g = \Delta_g^2 - \textmd{div}_g(a_n R_g g + b_n Ric_g) d + \frac{n-4}{2} Q_g 
$$
with  $a_n=\frac{(n-2)^2 +4}{2(n-1)(n-2)}$ and  $b_n=-\frac{4}{n-2}$.  When $n \geq 5$, the operator $P_2^g$ is conformally invariant:    if $\tilde{g}=v^{\frac{4}{n-4}} g$ is a conformal metric to $g$, then for all $\phi\in C^\infty(M)$ we have
\begin{equation}\label{PQ01}
P_2^g(v \phi)  = v^{\frac{n+4}{n-4}} P_2^{\tilde{g}}(\phi). 
\end{equation}
As a consequence one has the following conformal transformation law
\begin{equation}\label{PQ02}
P_2^g(v)  = \frac{n-4}{2} Q_{\tilde{g}} v^{\frac{n+4}{n-4}}, ~~~~   \tilde{g}=v^{\frac{4}{n-4}} g. 
\end{equation} 
Therefore,   as in the Nirenberg problem,  the problem of  prescribing $Q$-curvature on $(\mathbb{S}^n, g_{\mathbb{S}^n})$ for  $n \geq 5$  can be reduced to the study of existence of positive solutions to Eq. \eqref{S001} with $m=2$.    The  problem has been studied in \cite{DH,DMA01,DMA02,Fe,R} and many others.  More generally,  higher order conformally invariant differential operators  on Riemannian manifolds  and their associated $Q$-curvatures  have also been studied in  \cite{GJMS,GZ,J} and so on.     In particular,  the operator  $P_m$ (or more precisely, $P_m^{g_{\mathbb{S}^n}}$) on $\mathbb{S}^n$ involved in Eq. \eqref{S001} is  the most typical example,  and prescribing its associated higher order $Q$-curvature  on $\mathbb{S}^n$  is  naturally  reduced to the study of  Eq. \eqref{S001}  with   $m\geq 3$.    One may  see the work of Jin-Li-Xiong \cite{JLX} for a  unified  approach for  all  $m \in (0, n/2)$.

 When $K$ is identically a positive constant on $\mathbb{S}^n$,   the fact that  {\it no}  solution of Eq. \eqref{S001} exists which is singular  at  one  point  has   been shown by Lin \cite{Lin1998} for  $m=2$ and by Wei-Xu \cite{Wei1999} for $m\geq 3$.  In this paper,  we study the existence of positive functions  $K\in C^1(\mathbb{S}^n)$ such that  Eq. \eqref{S001} has  a singular positive solution $v$  whose singular set consists of a {\it single}  point when $m\geq 2$.   Clearly  this  will be in contrast to the  results of Lin  \cite{Lin1998} and Wei-Xu \cite{Wei1999}.   Moreover, we also investigate whether such a solution  $v$  can always be constructed to be arbitrarily large near its singularity. This means that,  for any given large continuous function $\varphi: (0, 1) \to (0, \infty)$  whether  such a solution  $v$  can  be constructed to  satisfy  
$$
v(\xi) \neq  O(\varphi(|\xi - \xi_0|)) ~~~~ \textmd{as} ~ \xi \to \xi_0,
$$
where $\{\xi_0\}$ is the singular set of $v$ on $\mathbb{S}^n$.  

The main result of this paper is the following theorem.

\begin{theorem}\label{T-001}
Let $\varepsilon$ be a positive number  and let   $\varphi: (0, 1) \to (0, \infty)$ be a  continuous function.   Suppose that  $k \in C^1(\mathbb{S}^n)$  is a  positive function,  $m\geq 1$ is an integer  and  $n \geq 2m +4$.  Then there exists    $\xi_0 \in \mathbb{S}^n$,  and  a  positive function $K \in C^1(\mathbb{S}^n)$ satisfying 
\begin{equation}\label{AApp}
\|K - k\|_{C^1(\mathbb{S}^n)} < \varepsilon~~~~~ \textmd{and} ~~~~~K(\xi)=k(\xi) ~~ \textmd{for} ~~ |\xi - \xi_0| \geq \varepsilon
\end{equation} 
such that  Eq. \eqref{S001} has  a  positive solution $v\in C^{2m} (\mathbb{S}^n \backslash  \{\xi_0\})$   satisfying 
\begin{equation}\label{Ar-La}
v(\xi) \neq  O(\varphi(|\xi - \xi_0|)) ~~~~ \textmd{as} ~ \xi \to \xi_0. 
\end{equation}     
\end{theorem} 

\begin{remark}\label{Rem-01}
The solution   $v$  constructed  in Theorem \ref{T-001} is a  distributional  solution of \eqref{S001}  on the whole sphere $\mathbb{S}^n$.  That is,  $v \in L^{\frac{n+2m}{n-2m}}(\mathbb{S}^n)$  and $v$ satisfies 
\begin{equation}\label{DIS-01}
\int_{\mathbb{S}^n}  v P_m (\phi) = \int_{\mathbb{S}^n} K v^{\frac{n+2m}{n-2m}} \phi~~~~~ \textmd{for} ~ \textmd{all} ~ \phi \in C^\infty(\mathbb{S}^n). 
\end{equation}
At first glance,  the above statement  seems impossible since the solution  $v$  in Theorem \ref{T-001} satisfies \eqref{Ar-La} where no  bound is imposed on the size of $\varphi$ near $0$.  However this is not the case.  Indeed, if $v$ is a $C^{2m}$ positive solution of Eq. \eqref{S001} in some punctured neighborhood $\Omega \backslash  \{\xi_0\}$ of some point $\xi_0\in \mathbb{S}^n$,   then $v\in L_{\textmd{loc}}^{\frac{n+2m}{n-2m}} (\Omega)$ and $u$ is a distributional  solution of \eqref{S001}  in $\Omega$. See, e.g., \cite{Jin2019,Y20}.   
\end{remark}

To prove Theorem \ref{T-001}, take  $\xi_0 \in \mathbb{S}^n$ such that   $\nabla k(\xi_0)=0$ and let $\pi$ be the stereographic projection  of $\mathbb{S}^n$ onto $\mathbb{R}^n \cup \{\infty\}$ with  $\xi_0$ being  the  south pole.   Then  $v$ is a positive solution of \eqref{S001} with singular set $\{\xi_0\}$ if  and only if 
$$
u(x):= \left( \frac{2}{1 + |x|^2} \right)^{\frac{n-2m}{2}}  v (\pi^{-1} (x)),~~~~~ x \in \mathbb{R}^n \backslash \{0\}
$$
is a positive solution of
$$
\begin{cases}
(-\Delta)^m u = K(x)  u^\frac{n+2m}{n-2m} ~~~~~~ & \textmd{in} ~ \R^n \backslash \{0\}, \\
u (x) \neq  O(1)   ~~~~~~ & \textmd{as} ~ |x| \to 0^+, \\
u (x) = O(|x|^{2m -n})   ~~~~~~ & \textmd{as} ~ |x| \to \infty. 
\end{cases} 
$$
Therefore, to show Theorem \ref{T-001}, it suffices to establish  the following result on the equation 
\begin{equation}\label{Rn-Kx}
(-\Delta)^m u = K(x)  u^\frac{n+2m}{n-2m} ~~~~~~  \textmd{in} ~ \R^n \backslash \{0\}. 
\end{equation}

\begin{theorem}\label{thmC1K} 
Suppose that  $m\geq 1$ is an integer and $n \geq  2m +4$.  Let $k:\R^n \to \R$ be a $C^1$ function which is bounded between two positive constants and satisfies $\nabla k(0) = 0$. Let $\varepsilon$ be a positive number and $\varphi:(0,1) \to (0,\infty)$ be a continuous function. Then there exists a $C^1$ positive function $K:\R^n \to \R$ satisfying $ \nabla K(0) = 0$,  $K(0) = k(0)$,  $K(x) = k(x)$ for $|x| \geq \varepsilon$ and 
\begin{equation}\label{KkC1}
\|K - k\|_{C^1(\R^n)} < \varepsilon
\end{equation}
such that Eq. \eqref{Rn-Kx}  has  a $C^{2m}$ positive solution $u$    satisfying
\begin{equation}\label{uneqphi}
u(x) \neq O (\varphi(|x|)) ~~~~~~ \textmd{as} ~ |x| \to 0^+, 
\end{equation}
\begin{equation}\label{ubehainf}
u(x) = O (|x|^{2m-n}) ~~~~~~ \textmd{as} ~ |x| \to \infty, 
\end{equation}
and 
\begin{equation}\label{Sign-01}
(-\Delta)^s u > 0 ~~~~~~ \textmd{in} ~ \R^n \backslash  \{0\},  ~~~  s = 1, 2, \dots, m-1. 
\end{equation} 
\end{theorem}

\begin{remark} When $m \geq 2$ and $n = 2m + 1$, $2m + 2$ or $2m + 3$, the existence of such $K$ is still unknown.   When $m=1$,  Theorem \ref{thmC1K}  is not true in dimension  $n=3$ (see Chen-Lin \cite{Chen1997E})  or $n=4$  (see Taliaferro-Zhang \cite{TaliaferroZhang2006}),  but  its validity  in  dimension  $n=5$ is  unknown.  
\end{remark}

For when $m=1$, Theorem \ref{thmC1K} was proved  by Taliaferro \cite{Taliaferro2005}.    In the case of $m=1$,  we also mention that Taliaferro \cite {Taliaferro1999} proved the existence of positive functions $K \in C(\mathbb{R}^n\backslash \{0\})$ with $0 < a < K(x) < b$ in $\mathbb{R}^n\backslash \{0\}$ 
such that  Eq. \eqref{Rn-Kx} has a  $C^2$ positive solution $u$ satisfying \eqref{uneqphi} in  dimension  $n \geq 3$,  where $a$ and $b$ are given positive constants satisfying $b/a  > 2^{4/(n-2)}$,  and Taliaferro-Zhang \cite{TaliaferroZhang2003} proved  the existence of positive continuous functions $K \in C(\mathbb{R}^n)$ such that  Eq. \eqref{Rn-Kx} has a positive  solution satisfying \eqref{uneqphi}.   In \cite{Leung2003},   for  $m=1$,  Leung proved the existence of a positive  Lipschitz  continuous function $K$ on $\mathbb{R}^n$,  $n \geq 5$,  such that Eq. \eqref{Rn-Kx}  has a $C^2$ positive solution $u$ {\it not}  satisfying $u(x) = O(|x|^{-(n-2)/2})$  near the origin.

On the other direction,   Jin and Xiong in \cite{Jin2019}  recently proved that every positive solution $u$ of the equation 
\begin{equation}\label{JX} 
\begin{cases}
(-\Delta)^m u =   u^\frac{n+2m}{n-2m},    & n > 2m  \\
(-\Delta)^s u \geq  0,     & s = 1, 2, \dots, m-1 
\end{cases} 
~~~~~~ \textmd{in} ~ B_1 \backslash \{0\} 
\end{equation} 
satisfies the a priori bound 
\begin{equation}\label{JX-001} 
u(x) = O(|x|^{-\frac{n-2m}{2}}) ~~~~~~ \textmd{as} ~ |x| \to 0^+. 
\end{equation}
Moreover, they also showed that every positive solution $u$ of \eqref{JX} is   asymptotically  radially  symmetric near the origin.     Clearly,  the conclusion of Theorem \ref{thmC1K}  that the solution $u$ can be constructed  to  satisfy \eqref{uneqphi} and \eqref{Sign-01}   is also in contrast to the result of Jin and Xiong.

We will prove Theorem \ref{thmC1K} in the next section.  Our  proof  adapts that of Taliaferro \cite{Taliaferro2005} when $m=1$.  The basic  idea is as follows. Without loss of generality we may assume $k(0)=1$.  Let  
$$ 
w_\lambda(x) =c_{n,m} \left( \frac{\lambda}{\lambda^2 + |x|^2} \right)^\frac{n-2m}{2}
$$
be a smooth positive solution (which is usually called a {\it bubble})  of the equation  $(-\Delta)^m u =u^\frac{n+2m}{n-2m}$ on  $\R^n$   for a positive constant  $c_{n,m} >0$ and for each $\lambda >0$.   Notice  that  as $\lambda \to 0^+$,  $w_\lambda(x)$ and each of its partial derivatives with respect to $x$ converge uniformly to zero  on every closed subset of $\mathbb{R}^n \backslash \{0\}$ and $w_\lambda(0)$ tends to $+\infty$. Define $u_i(x)=w_{\lambda_i}(x - x_i)$, where $\{x_i\}_{i=1}^{\infty}$ is a sequence of distinct points in $B_\delta \backslash \{0\}$ for some small $\delta \in (0, \varepsilon)$ which tends to the origin  and $\{\lambda_i\}_{i=1}^{\infty}$  is a sequence of positive numbers which tends sufficiently fast to $0$.    Then the function $\tilde{u}:=\sum_{i=1}^\infty u_i  \in C^\infty(\mathbb{R}^n \backslash \{0\})$ would satisfy $\tilde{u}(x) \neq O (\varphi(|x|))$  as  $|x| \to 0^+$  and  {\it approximately} satisfy   
$$
(-\Delta)^m \tilde{u} = k(x) \tilde{u}^\frac{n+2m}{n-2m} ~~~~~ \textmd{in} ~   B_\delta \backslash \{0\}.
$$
We then  construct  an  appropriate  positive bounded function $u_0 \in C^{2m} (\mathbb{R}^n \backslash \{0\})  $ such that 
\begin{equation}\label{De-KK03}
u:= u_0 + \tilde{u}~~~~~ \textmd{and} ~~~~~ K:= \frac{(-\Delta)^m u}{u^{(n+2m) / (n-2m)}} 
\end{equation} 
satisfy the conclusion of Theorem \ref{thmC1K}.  To this end,   the sequences $\{x_i\}$ and $\{\lambda_i\}$ need to be selected very carefully.    We  will check  at the end of the proof that  $K$ defined in \eqref{De-KK03}  is $C^1$ on the whole space $\mathbb{R}^n$,  where it  becomes  clear why  we  need $n \geq 2m +4$.  For the higher order equation \eqref{Rn-Kx},  we  need  to  establish several more delicate estimates on the gradient  of $K$.   Indeed,   if one completely follows the estimates for  the gradient  of $K$   in   Taliaferro \cite{Taliaferro2005}, then the stronger condition $n \geq 6m$ will  be required.      
               
\vskip0.10in  

\noindent{\bf Acknowledgments.}   Both authors  would like to thank Prof. Tianling Jin for many helpful discussions and encouragement.

\section{Proof of Theorem \ref{thmC1K}}

We will use $B_r(x)$ to denote the open ball of radius $r$ in $\R^n$ with center $x$ and write $B_r(0)$ as $B_r$ for short. We write $a_i \sim b_i$ if the sequence $\{a_i / b_i\}_{i=1}^\infty$ is bounded between two positive constants depending only on $n$, $m$, $\inf_{\R^n} k$ and $\sup_{\R^n} k$. 
To prove Theorem \ref{thmC1K},  we also need the following simple lemma.  
\begin{lemma}{(\cite{Taliaferro2005})} \label{Taliaferro} 
Suppose $\lambda > 1$, $\{a_i\}_{i=1}^N \subset (0,\infty)$, and $a_1 \geq a_i$ for $2 \leq i \leq N$. Then
$$
\frac{\sum_{i=1}^N a_i^\lambda}{\left( \sum_{i=1}^N a_i \right)^\lambda} \leq \frac{1 +\frac{a_2}{a_1}}{ 1 + \lambda \frac{a_2}{a_1} } < 1.
$$
\end{lemma}

{\noindent}{\bf  Proof of Theorem \ref{thmC1K}.} The proof consists of six steps.

\vskip0.10in

{\it Step 0. Preliminaries.}  Without loss of generality, we can assume that $0 < \varepsilon < 1$ and $k(0) = 1$. Since $\nabla k(0) = 0$, there exists a $C^1$ positive function $\tilde{k} : \R^n \to \R$ such that $\tilde{k} \equiv 1$ in a small neighborhood of the origin, $\tilde{k}(x) = k(x)$ for $|x| \geq \varepsilon$ and $\|\tilde{k} - k\|_{C^1(\R^n)} < \varepsilon/2$. Replacing $k$ by $\tilde{k}$, we can assume that $k \equiv 1$ in $B_\delta$ for some $0 < \delta < \varepsilon$.

Let
$$
\psi(r,\lambda) = c_{n,m} \left( \frac{\lambda}{\lambda^2 + r^2} \right)^\frac{n-2m}{2}
$$
with  $c_{n,m} = [(n+2m-2)!!/(n-2m-2)!!]^{(n-2m)/(4m)}$. It is easy to check that for every $\lambda > 0$, the function $w_\lambda(x) := \psi(|x|,\lambda)$ satisfies
$$
\begin{cases}
(-\Delta)^m w_\lambda = w_\lambda^\frac{n+2m}{n-2m}\\
(-\Delta)^s w_\lambda > 0, ~ s = 1,2,\dots,m-1
\end{cases} ~~~~~~ \textmd{in} ~ \R^n.
$$
After some calculations, one can find that there exist  $\delta_1$ and $\delta_2$ satisfying
\begin{equation}\label{defdelta}
0 < \delta_2 < \frac{\delta_1}{2} < \frac{\delta}{4}
\end{equation}
and for any $|x| \leq \delta_2$ or $|x| \geq \delta$,
\begin{equation}\label{12xx1xx22}
\frac{1}{2} < \frac{w_\lambda(x-x_1)}{w_\lambda(x-x_2)} < 2 ~~~~~~ \textmd{when} ~ |x_1| = |x_2| = \delta_1 ~ \textmd{and} ~ 0 < \lambda \leq \delta_2.
\end{equation}

Recall that  $k$ is bounded between two positive constants, we denote 
\begin{equation}\label{defab}
a = \frac{1}{2} \inf\limits_{\R^n} k ~~~~~ \textmd{and} ~~~~~ b = \sup\limits_{\R^n} k.
\end{equation}

Let $i_0$ be the smallest integer greater than $2$ such that
\begin{equation}\label{defi0}
i_0^\frac{4m}{n-2m} > \frac{ 2^\frac{3n+2m}{n-2m} }{ (2a)^\frac{n+2m}{4m} }.
\end{equation}

Choose a sequence $\{x_i\}_{i=1}^\infty$ of distinct points in $\R^n$ and a sequence $\{r_i\}_{i=1}^\infty$ of positive numbers such that
\begin{equation}\label{defxi}
|x_1| = |x_2| = \cdots = |x_{i_0}| = \delta_1, ~~~~~~ \lim\limits_{i \to \infty} |x_i| = 0,
\end{equation}
\begin{equation}\label{defri}
r_1 = r_2 = \cdots = r_{i_0} = \frac{\delta_2}{2}, ~~~~~~ \lim\limits_{i \to \infty} r_i = 0,
\end{equation}
\begin{equation}\label{defBi1}
B_{4r_i}(x_i) \subset B_{\delta_2} \setminus \{0\} ~~~~~~ \textmd{for} ~ i > i_0
\end{equation}
and
\begin{equation}\label{defBi2}
\overline{B_{2r_i}(x_i)} \cap \overline{B_{2r_j}(x_j)} = \emptyset ~~~~~~ \textmd{for} ~ j > i >i_0.
\end{equation}
In addition, we require that the union of the line segments $\overline{x_1x_2}$, $\overline{x_2x_3}$, $\dots$, $\overline{x_{i_0-1}x_{i_0}}$, $\overline{x_{i_0}x_1}$ be a regular $i_0$-gon. We will prescribe the side length of this polygon later. From \eqref{defdelta}, \eqref{defxi} and \eqref{defri} we know that
$$
\overline{B_{2r_i}(x_i)} \subset B_{2\delta_1} \setminus \bar{B}_{\delta_2} ~~~~~~ \textmd{for} ~ 1 \leq i \leq i_0
$$
and hence by \eqref{defBi1},
\begin{equation}\label{defBi3}
\overline{B _{2r_i}(x_i)} \cap \overline{B_{2r_j}(x_j)} = \emptyset  ~~~~~~ \textmd{for} ~ 1 \leq i \leq i_0 < j.
\end{equation}

Define  three functions  $f : [0,\infty) \times (0,\infty) \times (0,\infty) \to \R$ and $M, Z : (0,1) \times (0,\infty) \to (0,\infty)$ by
$$
f(z_1,z_2,z_3) = z_2(z_1 + z_3)^\frac{n+2m}{n-2m} - z_1^\frac{n+2m}{n-2m},
$$
\begin{equation}\label{defMZ}
M(z_2,z_3) = \frac{ z_2z_3^\frac{n+2m}{n-2m} }{ \left( 1 - z_2^\frac{n-2m}{4m} \right)^\frac{4m}{n-2m} } ~~~~~~ \textmd{and} ~~~~~~ Z(z_2,z_3) = \frac{ z_3 z_2^\frac{n-2m}{4m} }{ 1 - z_2^\frac{n-2m}{4m} }.
\end{equation}
For  each  fixed $(z_2,z_3) \in (0,1) \times (0,\infty)$, the function $f(\cdot,z_2,z_3) : [0,\infty) \to \R$ is strictly increasing on $[0,Z(z_2,z_3)]$ and is strictly decreasing on $[Z(z_2,z_3),\infty)$,  and attains  its maximum value $M(z_2,z_3)$ at $z_1 = Z(z_2,z_3)$.

Define $F : [0,\infty) \times (0,\infty) \times (0,\infty) \to (0,\infty)$ by
$$
F(z_1,z_2,z_3) = 
\begin{cases}
f(z_1,z_2,z_3) & \textmd{if} ~ 0 < z_2 < 1 ~ \textmd{and} ~ 0 \leq z_1 \leq Z(z_2,z_3), \\
M(z_2,z_3) & \textmd{if} ~ 0 < z_2 < 1 ~ \textmd{and} ~ z_1 > Z(z_2,z_3), \\
f(z_1,z_2,z_3) ~~~~~~ & \textmd{if} ~ z_2 \geq 1.
\end{cases}
$$
It is easy to see that $f$ and $F$ are $C^1$, $f \leq F$ and $F$ is non-decreasing in $z_1$, $z_2$ and $z_3$.

\vskip0.10in

{\it Step 1.  Selecting  the sequences $\{x_i\}$ and $\{\lambda_i\}$.}   Let 
$$
w(x) = (2b)^{-\frac{n}{2m}} \psi(|x|,1) = \frac{c_{n,m}}{(2b)^\frac{n}{2m}} \left( \frac{1}{1 + |x|^2} \right)^\frac{n-2m}{2} ~~~~~~ \textmd{for} ~ x \in \R^n.
$$
Then we have
\begin{equation}\label{Lmw}
(-\Delta)^m w = (2b)^\frac{2n}{n-2m} w^\frac{n+2m}{n-2m} ~~~~~~ \textmd{in} ~ \R^n. 
\end{equation}

Choose a sequence $\{\varepsilon_i\}_{i=1}^\infty$ of positive numbers such that
\begin{equation}\label{defepsiloni}
\varepsilon_1 = \varepsilon_2 = \cdots = \varepsilon_{i_0} ~~~~~~ \textmd{and} ~~~~~~ \varepsilon_i \leq 2^{-i} ~~~~~~ \textmd{for}~i \geq 1.
\end{equation}

Now we introduce four sequences of real numbers  as follows.   For $i \geq 1$, let
\begin{equation}\label{defki}
k_i \in \left( \frac{1}{2},1 \right) ~~~~~~ \textmd{with} ~~~~~~ k_1 = k_2 = \cdots = k_{i_0},
\end{equation}
\begin{equation}\label{defMi}
M_i = \frac{ M(k_i,2w(0)) }{ (2w(0))^\frac{n+2m}{n-2m} } = \frac{k_i}{ \left( 1 - k_i^\frac{n-2m}{4m} \right)^\frac{4m}{n-2m} },
\end{equation}
\begin{equation}\label{defrhoi}
\rho_i = \sup \left\{ \rho > 0 ~ : ~ I_{2m} (\chi_{B_{2\rho_i}(x_i)}) \leq \frac{w}{2^{i+1}( 2w(0) )^\frac{n+2m}{n-2m} M_i } \right\}
\end{equation}
and
\begin{equation}\label{deflambdai}
\lambda_i = \sup \{ \lambda > 0 ~ : ~ w_\lambda(x-x_i) \leq \varepsilon_i a^\frac{n-2m}{4m} w(x) ~~ \textmd{for} ~ |x-x_i| \geq \rho_i \}, 
\end{equation}
where $I_{2m}$ is the Riesz potential operator of order $2m$   and $\chi_{B_{2\rho_i}(x_i)}$ is the characteristic function of the ball $ B_{2\rho_i}(x_i)$. Then  we have
\begin{lemma}\label{sim} For $i\geq 1$, 
$$
M_i \sim \frac{1}{(1-k_i)^\frac{4m}{n-2m}}, ~~~~ \rho_i^{2m} \sim \frac{1}{2^i M_i} ~~~~\textmd{and} ~~~~  \lambda_i \sim \varepsilon_i^\frac{2}{n-2m}\rho_i^2.
$$
\end{lemma}
\begin{proof} Since $k_i \in (0, 1/2)$, the first estimate is obvious. For the second, we only need to notice that
$$
\frac{1}{C} \rho_i^{2m} \leq I_{2m}( \chi_{B_{2\rho_i}(x_i)} ) \leq C \rho_i^{2m} ~~~~~~ \textmd{for} ~ |x-x_i| \leq 2 \rho_i
$$
and
$$
\frac{1}{C} \rho_i^n |x-x_i|^{2m-n} \leq I_{2m}( \chi_{B_{2\rho_i}(x_i)} ) \leq C \rho_i^n |x-x_i|^{2m-n} ~~~~~~ \textmd{for} ~ |x-x_i| \geq 2 \rho_i
$$
for some positive constant $C$ depending on $n$ and $m$.

For the last, the inequality $w_\lambda(x-x_i) \leq \varepsilon_i a^\frac{n-2m}{4m} w(x)$ holds for $|x-x_i| \geq \rho_i$ if and only if
$$
\lambda (1 + |x|^2) \leq \varepsilon_i^\frac{2}{n-2m} a^\frac{1}{2m} (2b)^{ -\frac{n}{m(n-2m) }} (\lambda^2 + |x-x_i|^2) ~~~~~~ \textmd{for} ~ |x-x_i| \geq \rho_i.
$$
After some rotations, the above inequality is equivalent to
$$
\varepsilon_i^\frac{2}{n-2m} a^\frac{1}{2m} (2b)^{ -\frac{n}{m(n-2m)} } (\lambda^2 + |y'|^2 + | y_1-|x_i| |^2) \geq \lambda (1 + |y'|^2 + |y_1|^2), 
$$
where $y_1 \in \R$, $y' \in \R^{n-1}$ and $| y_1-|x_i| |^2 + |y'|^2 \geq \rho_i^2$.  Comparing the coefficients of $|y'|^2$, we get
$$
\varepsilon_i^\frac{2}{n-2m} a^\frac{1}{2m} (2b)^{-\frac{n}{m(n-2m)}} \geq \lambda.
$$
We also need $\lambda$ satisfying
$$
\varepsilon_i^\frac{2}{n-2m} a^\frac{1}{2m} (2b)^{ -\frac{n}{m(n-2m)} } (\lambda^2 + | y_1-|x_i||^2) \geq \lambda (1 + |y_1|^2) \geq 0 ~~~~~~ \textmd{for} ~ | y_1-|x_i| | \geq \rho_i.
$$
This inequality holds if and only if
$$
\begin{cases}
\varepsilon_i^\frac{2}{n-2m} a^\frac{1}{2m} (2b)^{ -\frac{n}{m(n-2m)} } \rho_i^2 \geq \lambda [1 + (|x_i| + \rho_i)^2 - \varepsilon_i^\frac{2}{n-2m} a^\frac{1}{2m} (2b)^{ -\frac{n}{m(n-2m)} } \lambda],\\
\varepsilon_i^\frac{2}{n-2m} a^\frac{1}{2m} (2b)^{ -\frac{n}{m(n-2m)} } \rho_i^2 \geq \lambda [1 + (|x_i| - \rho_i)^2 - \varepsilon_i^\frac{2}{n-2m} a^\frac{1}{2m} (2b)^{ -\frac{n}{m(n-2m)} } \lambda].
\end{cases}
$$
Since both $1 + (|x_i| + \rho_i)^2 - \varepsilon_i^\frac{2}{n-2m} a^\frac{1}{2m} (2b)^{ -\frac{n}{m(n-2m)} } \lambda$ and $1 + (|x_i| - \rho_i)^2 - \varepsilon_i^\frac{2}{n-2m} a^\frac{1}{2m} (2b)^{ -\frac{n}{m(n-2m)} } \lambda$ are bounded between two positive constants depending only on $n$, $m$, $a$ and $b$, we get the desired estimate.
\end{proof}  

By Lemma \ref{sim}, after increasing the values of $k_i$ for certain values of $i$ while holding $\varepsilon_i$ fixed, we can assume for $i \geq 1$ that 
\begin{equation}\label{3seq1}
M_i > 3^i, ~~~~~~ \rho_i \in (0,r_i), ~~~~~~ \lambda_i \in (0,\delta_2),
\end{equation}
and
\begin{equation}\label{3seq2}
k_i^\frac{n+2m}{4m} > \frac{ 1 + \left( \frac{1}{3} \right)^{n-2m} }{ 1 + \frac{n+2m}{n-2m} \left( \frac{1}{3} \right)^{n-2m} }, ~~~~~~ M_i^\alpha > \max \left\{ \frac{1}{ \varepsilon_i^\frac{4m}{n-2m} }, ~ 2^i \right\}, ~~~~~~ \lambda_i^\alpha < \frac{ \varepsilon_i^\frac{2m}{n-2m} }{2^i}, 
\end{equation}
where $\alpha = \alpha(m) \in (0,1/6)$ is a constant to be specified later.

Notice that for $1 \leq i \leq i_0$, $\rho_i$ and $\lambda_i$ do not change as $x_i$ moves on the sphere $|x| = \delta_1$. Therefore, we can require that the union of the line segments $\overline{x_1x_2}$, $\overline{x_2x_3}$, $\dots$, $\overline{x_{i_0-1}x_{i_0}}$, $\overline{x_{i_0}x_1}$ be a regular $i_0$-gon with side length $4\rho_1$. Thus,
\begin{equation}\label{distBiBj}
\textmd{dist} ( B_{\rho_i}(x_i),B_{\rho_j}(x_j) ) \geq \rho_i + \rho_j ~~~~~~ \textmd{for} ~ 1 \leq i < j \leq i_0.
\end{equation}
By \eqref{defBi2}, \eqref{defBi3} and \eqref{3seq1}, the inequality \eqref{distBiBj} also holds for $1 \leq i < j$.

For $i \geq 1$, define
$$
u_i(x) := w_{\lambda_i}(x-x_i).
$$
Then one can check that
\begin{equation}\label{minBjww}
\min\limits_{ x \in B_{2\rho_j}(x_j) } \frac{ u_{j+1}(x) }{ u_{j-1}(x) } > \left( \frac{1}{3} \right)^{n-2m} ~~~~~~ \textmd{for} ~ 2 \leq j \leq i_0 - 1
\end{equation}
and a similar inequality holds when $j$ is $1$ or $i_0$.

Here we  also give  some inequalities which will be used later. By \eqref{3seq2} and Lemma \ref{sim} we have  for $1 \leq j \leq i_0$ that 
\begin{equation}\label{minZkjsumMj}
\aligned
& \min\limits_{ x\in B_{2\rho_j}(x_j) } Z \left( k_j^\frac{n+2m}{4m} , \sum\limits_{i=1,i\neq j}^{i_0} u_i(x) \right)\\
= & \min\limits_{ x\in B_{2\rho_1}(x_1) } Z \left( k_1^\frac{n+2m}{4m} , \sum\limits_{i=2}^{i_0} u_i(x) \right)\\
\geq & \min\limits_{ x\in B_{2\rho_1}(x_1) } Z \left( k_1^\frac{n+2m}{4m} , u_2(x) \right)\\
\geq & Z \left( k_1^\frac{n+2m}{4m} , \psi(6\rho_2,\lambda_2) \right) \sim \frac{1}{1-k_1} \left( \frac{\lambda_1}{\lambda_1^2+(6\rho_1)^2} \right)^\frac{n-2m}{2} \sim \frac{1}{1-k_1} \left( \frac{\lambda_1}{\rho_1^2} \right)^\frac{n-2m}{2}\\
\sim & \frac{\varepsilon_1}{1-k_1} \geq \frac{1}{(1-k_1) M_1^\frac{\alpha(n-2m)}{4m}} \sim M_1^\frac{(1-\alpha)(n-2m)}{4m} = M_j^\frac{(1-\alpha)(n-2m)}{4m}.
\endaligned
\end{equation}
Thus, by increasing $k_1$ (recall that $k_1 = k_2 = \cdots = k_{i_0}$)  we have
\begin{equation}\label{minZkjijw0}
\min\limits_{x\in B_{2\rho_j}(x_j)} Z \left( k_j^\frac{n+2m}{4m}, \sum\limits_{i=1,i\neq j}^{i_0} u_i(x) \right) > w(0) ~~~~~~ \textmd{for} ~ 1 \leq j \leq i_0.
\end{equation}
By Lemma \ref{sim},
\begin{equation}\label{ZkjMj}
Z \left( k_j^\frac{n+2m}{4m} , \frac{1}{ 2M_j^\frac{\alpha(n-2m)}{4m} } \right) \sim \frac{1}{1-k_j} \frac{1}{ M_j^\frac{\alpha(n-2m)}{4m} }\sim M_j^\frac{(1-\alpha)(n-2m)}{4m} ~~~~~~ \textmd{for} ~ j \geq 1.
\end{equation}
Therefore, by increasing each term of the sequence $\{k_j\}_{j=1}^\infty$, we also have
$$
Z \left( k_j^\frac{n+2m}{4m} , \frac{1}{ 2M_j^\frac{\alpha(n-2m)}{4m} } \right) > w(0) ~~~~~~ \textmd{for} ~ j \geq 1.
$$
Then,  by \eqref{deflambdai}, \eqref{defab} and \eqref{defepsiloni} we have for $j\geq 1$ and $|x-x_j| \geq \rho_j$ that 
\begin{equation}\label{wxjljw0Z}
\aligned
u_j(x) & \leq \varepsilon_j a^\frac{n-2m}{4m} w(0)\\
& < w(0) < Z \left( k_j^\frac{n+2m}{4m} , \frac{1}{ 2M_j^\frac{\alpha(n-2m)}{4m} } \right).
\endaligned
\end{equation}

It follows from Lemma \ref{sim} and \eqref{defepsiloni} that
\begin{equation}\label{maxddspsi}
\max\limits_{ s \geq \rho_j } \left| \frac{d}{ds}( \psi(s,\lambda_j) ) \right| \sim \left( \frac{\lambda_j}{\rho_j^2} \right)^\frac{n-2m}{2} \frac{1}{\rho_j} \sim \varepsilon_j 2^\frac{j}{2m} M_j^\frac{1}{2m} < M_j^\frac{1}{2m} ~~~~~~ \textmd{for} ~ j \geq 1.
\end{equation}

\vskip0.10in

{\it Step 2.  Estimating the sum of the bubbles $u_i$.}     By \eqref{deflambdai} and \eqref{defepsiloni} we have 
\begin{equation}\label{uileqeiw}
u_i \leq \varepsilon_i a^\frac{n-2m}{4m} w ~~~~~~ \textmd{in} ~ \R^n \setminus B_{\rho_i}(x_i)
\end{equation}
and
\begin{equation}\label{sumuleqw}
\sum\limits_{i=1}^\infty u_i \leq a^\frac{n-2m}{4m} w ~~~~~~ \textmd{in} ~ \R^n - \bigcup\limits_{i=1}^\infty B_{\rho_i}(x_i).
\end{equation}
By \eqref{defepsiloni} and \eqref{uileqeiw}, we know that  $\sum\limits_{i=1}^\infty u_i \in C^\infty(\R^n\setminus\{0\})$ satisfies   
\begin{equation}\label{LmuiLsui}
\begin{cases}
(-\Delta)^m \left( \sum\limits_{i=1}^\infty u_i \right) = \sum\limits_{i=1}^\infty u_i^\frac{n+2m}{n-2m}\\
(-\Delta)^s \left( \sum\limits_{i=1}^\infty u_i \right) = \sum\limits_{i=1}^\infty (-\Delta)^s u_i > 0, ~ s = 1,2,\dots,m-1
\end{cases} ~~~~~~ \textmd{in} ~ \R^n \setminus \{0\}.
\end{equation}
By increasing $k_i$ for each $i$, we can assume that
\begin{equation}\label{defuixi}
u_i(x_i) = c_{n,m} \lambda_i^{ -\frac{n-2m}{2} } > i \varphi(|x_i|) ~~~~~~ \textmd{for} ~ i \geq 1
\end{equation} 
and $u_i + |\nabla u_i| < 2^{-i}$ in $\R^n \setminus B_{2r_i}(x_i)$, $i \geq 1$.  Thus  by \eqref{defBi2} and \eqref{defBi3},
\begin{equation}\label{ujnablauj2j}
u_i + |\nabla u_i| < 2^{-i} ~~~~~~ \textmd{in} ~ B_{2r_j}(x_j) 
\end{equation} 
when $i \neq j$ and either $i \geq 1$ and $j > i_0$ or $i > i_0$ and $1 \leq j \leq i_0$.  Again, by increasing $k_i$ for $i > i_0$, we can force $u_i$ and $M_i$ to satisfy
\begin{equation}\label{sumuminu1}
\sum\limits_{i = i_0 + 1}^\infty u_i(x) < \frac{1}{2} \min\limits_{1 \leq i \leq i_0} u_i(x) ~~~~~~ \textmd{for} ~ |x| \geq \delta_2,
\end{equation}
\begin{equation}\label{sumu12}
\sum\limits_{i = i_0 + 1, i \neq j}^\infty u_i(x) < \frac{u_1}{2} ~~~~~~ \textmd{in} ~ B_{2r_j}(x_j),~j > i_0
\end{equation}
and
\begin{equation}\label{Mjminu1}
\frac{1}{M_j^\frac{\alpha(n-2m)}{4m}} < \min\limits_{|x| \leq \delta} u_1(x) ~~~~~~ \textmd{for} ~ j > i_0.
\end{equation}

It follows from \eqref{ujnablauj2j} and \eqref{uileqeiw} that
\begin{equation}\label{sumuiuinC}
\sum\limits_{i=1, i \neq j}^\infty u_i + u_i^\frac{n+2m}{n-2m} \leq C ~~~~~~ \textmd{in} ~ B_{\rho_j}(x_j), ~ j \geq 1. 
\end{equation}
Similarly, by \eqref{ujnablauj2j}, \eqref{uileqeiw}, Lemma \ref{sim} and \eqref{defi0},
$$
\aligned
& \sum\limits_{i=1, i \neq j}^\infty |\nabla u_i| + u_i^\frac{4m}{n-2m} |\nabla u_i| \leq \sum\limits_{i=1, i \neq j}^{i_0} |\nabla u_i| + u_i^\frac{4m}{n-2m} |\nabla u_i| + C\\
& ~~~ \leq \sum\limits_{i=1, i \neq j}^{i_0} u_i \frac{1}{\rho_j} + u_i^\frac{n+2m}{n-2m} \frac{1}{\rho_j} + C\\
& ~~~ \leq C 2^\frac{j}{2m} M_j^\frac{1}{2m} \leq C 2^\frac{i_0}{2m} M_j^\frac{1}{2m} \leq C M_j^\frac{1}{2m} ~~~~~~ \textmd{in} ~ B_{\rho_j}(x_j), ~ 1 \leq j \leq i_0,  
\endaligned
$$
and by \eqref{ujnablauj2j} and \eqref{uileqeiw},
$$
\sum\limits_{i=1, i \neq j}^\infty |\nabla u_i| + u_i^\frac{4m}{n-2m} |\nabla u_i| \leq C ~~~~~~ \textmd{in} ~ B_{\rho_j}(x_j), ~ j > i_0. 
$$
Thus, we get
\begin{equation}\label{sumnablauiuinnablaui}
\sum\limits_{i=1, i \neq j}^\infty |\nabla u_i| + u_i^\frac{4m}{n-2m} |\nabla u_i| \leq C M_j^\frac{1}{2m} ~~~~~~ \textmd{in} ~ B_{\rho_j}(x_j), ~ j \geq 1. 
\end{equation}

\vskip0.10in

{\it Step 3.  Constructing  the correction function $u_0$.}     Since $n \geq 2m+4$, by Lemma \ref{sim} and \eqref{3seq1} we have
\begin{equation}\label{1kirhoi}
\frac{1-k_i}{\rho_i} \sim \frac{2^\frac{i}{2m} M_i^\frac{1}{2m}}{ M_i^\frac{n-2m}{4m} } \sim \frac{2^\frac{i}{2m}}{M_i^\frac{n-2m-2}{4m}} \leq \frac{2^\frac{i}{2m}}{M_i^\frac{1}{2m}}  \leq \left( \frac{2}{3} \right)^\frac{i}{2m} \to 0 ~~~~~~ \textmd{as} ~ i \to \infty.
\end{equation}
Let $\eta:[0,\infty) \to [0,1]$ be a $C^\infty$ cut-off function satisfying $\eta(t) = 1$ for $0 \leq t \leq 1$ and $\eta(t) = 0$ for $t \geq 3/2$. Define
\begin{equation}\label{defkappa}
\kappa(x) = k(x) + \sum\limits_{i=1}^\infty (k_i-k(x))\eta_i(x) ~~~~~~ \textmd{for} ~ x \in \R^n, 
\end{equation}
where $\eta_i(x) = \eta(|x-x_i|/\rho_i)$. Since $\{\eta_i\}_{i=1}^\infty$ have disjoint supports contained in $B_{2\delta_1} \setminus \{0\}$, $\kappa$ is well-defined. Recall that $k \equiv 1$ in $B_{\delta}$, we obtain $\kappa(0) = k(0) = 1$, $\kappa \leq k$ in $\R^n$ and $\kappa(x) = k(x)$ for $|x| \geq 2 \delta_1$. By \eqref{defab} and \eqref{defki} we have
\begin{equation}\label{infka}
\inf \limits_{\R^n} \kappa \geq a.
\end{equation}
Since $k \equiv 1$ in $B_\delta$ and then 
\begin{equation}\label{nablakappa}
\nabla \kappa(x) = \sum\limits_{i=1}^\infty \frac{k_i-1}{\rho_i} \eta'\left( \frac{|x-x_i|}{\rho_i} \right) \frac{x-x_i}{|x-x_i|} ~~~~~~ \textmd{for} ~ x \in B_\delta,
\end{equation}
it follows from \eqref{1kirhoi} that $\kappa \in C^1(\R^n)$ and $\nabla \kappa(0) = 0$.

By \eqref{defrhoi},
\begin{equation}\label{ImMw2}
0 < I_{2m} \overline{M} < \frac{w}{2} ~~~~~~ \textmd{in} ~ \R^n, 
\end{equation}
where
$$
\overline{M}(x):=
\begin{cases}
(2w(0))^\frac{n+2m}{n-2m} M_i ~~~~~~ & \textmd{in} ~ B_{\rho_i}(x_i), ~ i \geq 1,\\
0 ~~~~~~ & \textmd{in} ~ \R^n - \bigcup\limits_{i=1}^\infty B_{2\rho_i}(x_i),\\
(2w(0))^\frac{n+2m}{n-2m} M_i \left( 2-\frac{|x-x_i|}{\rho_i} \right) ~~~~~~ & \textmd{in} ~ B_{2\rho_i}(x_i) \setminus B_{\rho_i}(x_i), ~ i \geq 1. 
\end{cases}
$$
Since $\overline{M}$ is locally Lipschitz continuous in $\R^n \setminus \{0\}$, we have  $\overline{v} : = w/(2b) + I_{2m} \overline{M} \in C_{loc}^{2m,\beta}(\R^n \setminus \{0\})$ for any $0 < \beta < 1$. By \eqref{Lmw}, 
\begin{equation}\label{defoverv}
\begin{cases}
(-\Delta)^s \overline{v} = (-\Delta)^s w/(2b) + I_{2(m-s)} \overline{M} > 0, ~ s = 1,2,\dots,m-1\\
(-\Delta)^m \overline{v} = (2b)^\frac{n+2m}{n-2m} w^\frac{n+2m}{n-2m} + \overline{M}
\end{cases} ~~~~~~ \textmd{in} ~ \R^n \setminus \{0\}, 
\end{equation} 
where $I_{2(m-s)} $ is the Riesz potential operator of order $2(m-s)$.  It follows from \eqref{ImMw2} that
\begin{equation}\label{w2bvw}
\frac{w}{2b} < \overline{v} < w ~~~~~~ \R^n.
\end{equation}

Define $\underline{H}:\R^n \times [0,\infty) \to \R$ by
\begin{equation}\label{underH}
\underline{H}(x,v) = \kappa(x) \left( v + \sum\limits_{i=1}^\infty u_i(x) \right)^\frac{n+2m}{n-2m} - \sum\limits_{i=1}^\infty u_i(x)^\frac{n+2m}{n-2m}. 
\end{equation}
Then we have
$$
\underline{H}(x,v) = f(U(x), \kappa(x) , P(x,v)), 
$$
where
$$
U(x):= \left( \sum\limits_{i=1}^\infty u_i(x)^\frac{n+2m}{n-2m} \right)^\frac{n-2m}{n+2m} ~~~~~~ \textmd{and} ~~~~~~ P(x,v):= v + \sum\limits_{i=1}^\infty u_i(x) - U(x).
$$

Define $H:\R^n \times [0,\infty) \to (0,\infty)$ by
\begin{equation}\label{defH}
H(x,v) = F(U(x),  \kappa(x),  P(x,v)).
\end{equation}
Then
\begin{equation}\label{HleqM}
H(x,v) \leq M(\kappa(x), P(x,v)) ~~~~~~ \textmd{when} ~ \kappa(x) < 1.
\end{equation}
Moreover,  by the definition of $F$  we have that  $H(x,v) = \underline{H}(x,v)$ if and only if either $\kappa(x) < 1$ and $U(x) \leq Z(\kappa(x), P(x,v))$ or $\kappa(x) \geq 1$.

For $x \in \R^n - \bigcup\limits_{i=1}^\infty B_{\rho_i}(x_i)$ and $\kappa(x) < 1$, we have
$$
\aligned
U(x) & \leq \sum\limits_{i=1}^\infty u_i(x) \leq a^\frac{n-2m}{4m} w(x) ~~~~~~ \textmd{by} ~ \eqref{sumuleqw}\\
& \leq \frac{ w(x) \kappa(x)^\frac{n-2m}{4m} }{ 1 - \kappa(x)^\frac{n-2m}{4m} } ~~~~~~ \textmd{by} ~ \eqref{infka}\\
& \leq \frac{ P(x,w(x)) \kappa(x)^\frac{n-2m}{4m} }{ 1 - \kappa(x)^\frac{n-2m}{4m} }\\
& = Z( \kappa(x),P(x,w(x)) ) ~~~~~~ \textmd{by} ~ \eqref{defMZ}.
\endaligned
$$
Hence
$$
H(x,w(x)) = \underline{H}(x,w(x)) ~~~~~~ \textmd{for} ~ x \in \R^n - \bigcup\limits_{i=1}^\infty B_{\rho_i}(x_i).
$$
Thus, for $x \in (\R^n \setminus \{0\}) - \bigcup\limits_{i=1}^\infty B_{\rho_i}(x_i)$ and $0 \leq v \leq w(x)$  we have
\begin{equation}\label{HxvLmvx}
\aligned
H(x,v) & \leq H(x,w(x)) = \underline{H}(x,w(x)) \leq \kappa(x) \left( w(x) + \sum\limits_{i=1}^\infty u_i(x) \right)^\frac{n+2m}{n-2m}\\
& \leq b(2w(x))^\frac{n+2m}{n-2m} \leq (-\Delta)^m \overline{v}(x)
\endaligned
\end{equation}
by \eqref{sumuleqw}, \eqref{defab} and \eqref{defoverv}.

For $x \in B_{\rho_i}(x_i)$ and  $i \geq 1$ we have $\kappa(x) \equiv  k_i < 1$. Hence, from \eqref{HleqM} we obtain   for $x \in B_{\rho_i}(x_i)$ and $0 \leq v \leq w(x)$ that 
\begin{equation}\label{HxvLmvx2}
\aligned
H(x,v) & \leq M(k_i,P(x,v)) = \frac{ k_i P(x,v)^\frac{n+2m}{n-2m} }{\left( 1 - k_i^\frac{n-2m}{4m} \right)^\frac{4m}{n-2m}}\\
& \leq M_i \Big(  v + \sum\limits_{j=1}^\infty u_j(x)  -U(x) \Big)^\frac{n+2m}{n-2m} ~~~~~~ \textmd{by} ~ \eqref{defMi}\\
& \leq M_i \Big( v + \sum\limits_{j=1, j \neq i}^\infty  u_j(x) \Big)^\frac{n+2m}{n-2m}\\
& \leq M_i (2w(x))^\frac{n+2m}{n-2m} ~~~~~~ \textmd{by} ~ \eqref{uileqeiw}\\
& \leq M_i (2w(0))^\frac{n+2m}{n-2m} = \overline{M}(x) \leq (-\Delta)^m \overline{v}(x)   ~~~~~~ \textmd{by} ~ \eqref{defoverv}. 
\endaligned
\end{equation}
This  together with \eqref{HxvLmvx} implies that
\begin{equation}\label{supersol}
H(x,v) \leq (-\Delta)^m \overline{v}(x) ~~~~~~ \textmd{for} ~ x \in \R^n \setminus \{0\} ~ \textmd{and} ~ 0 \leq v \leq w(x).
\end{equation}

Hence, by the nonnegativity of $H$,   \eqref{w2bvw} and \eqref{supersol} we can use  $\underline{v} \equiv  0$ and $\overline{v}$ as sub- and super-solutions of the problem 
$$
(-\Delta)^m u = H(x,u) ~~~~~~ \textmd{in} ~ \R^n \setminus \{0\}.
$$
Now,  applying  the method of sub- and super-solutions  we  can get  the desired correction function $u_0$. 
\begin{lemma}\label{subsupermethod} There exists a $C^{2m}(\R^n \setminus \{0\})$ solution $u_0$ of 
\begin{equation}\label{defu0}
\begin{cases}
(-\Delta)^m u_0 = H(x,u_0)\\
0 \leq (-\Delta)^s u_0 \leq (-\Delta)^s \overline{v}, ~ s = 1,2,\dots,m-1\\
0 \leq u_0 \leq \overline{v} \leq w
\end{cases} ~~~~~~ \textmd{in} ~ \R^n \setminus \{0\}.
\end{equation}
\end{lemma}

\begin{proof}    
For each positive integer $l \geq 2$, we consider about the following problem  
\begin{equation}\label{localequ}
\begin{cases}
(-\Delta)^m v = H(x,v) ~~~~~~ & \textmd{in} ~ B_l \setminus \bar{B}_{1/l}, \\
(-\Delta)^s v = 0 & \textmd{on} ~ \partial(B_l \setminus \bar{B}_{1/l}), ~ s = 0,1,\dots,m-1.
\end{cases}
\end{equation}
Notice that $\underline{v}$, $\overline{v} \in C_{loc}^{2m,\beta}(\R^n \setminus \{0\})$ for any $0 < \beta < 1$, $H \in C^1(\R^n \setminus \{0\} \times [0,\infty))$ and $H$ is non-decreasing with respect to the last variable. By using the  method of sub- and super-solutions (see, e.g.,  Theorem 2.1 in \cite{Kusano1992}) the problem \eqref{localequ}  has  a  $C^{2m}$  solution $v_l$ satisfying $0 \leq v_l\leq \overline{v}$ and $0\leq (-\Delta)^s v_l \leq (-\Delta)^s \overline{v}$ for every $s = 1,2,\dots,m-1$.  It follows from standard elliptic theory that,   after passing to a subsequence, $\{v_l\}$  converges to a nonnegative function $u_0 \in C_{loc}^{2m}(\mathbb{R}^n \setminus \{0\})$ which satisfies  \eqref{defu0}.   
\end{proof}

{\it Step 4.  Defining  the solution $u$ and the function  $K$.}    Define $\overline{H}:\R^n \times [0,\infty) \to (0,\infty)$ by $\overline{H}(x,v) = F(U(x),k(x),P(x,v))$. Then $\underline{H} \leq H \leq \overline{H}$ since $\kappa \leq k$. In particular,
\begin{equation}\label{uHHoH}
\underline{H}(x,u_0(x)) \leq H(x,u_0(x)) \leq \overline{H}(x,u_0(x)) ~~~~~~ \textmd{for} ~ x \in \R^n \setminus \{0\}.
\end{equation}

For $|x| > \delta$ we have 
$$
\aligned
U(x)^\frac{n+2m}{n-2m} & = \sum\limits_{i=1}^\infty u_i(x)^\frac{n+2m}{n-2m}\\
& \leq i_0 2^\frac{n+2m}{n-2m} u_1(x)^\frac{n+2m}{n-2m} + u_1(x)^\frac{n+2m}{n-2m} ~~~~~~ \textmd{by} ~ \eqref{12xx1xx22} ~ \textmd{and} ~ \eqref{sumuminu1}\\
& \leq i_0 2^{ \frac{n+2m}{n-2m} + 1 } u_1(x)^\frac{n+2m}{n-2m} = \frac{ i_0^\frac{n+2m}{n-2m} }{ i_0^\frac{4m}{n-2m} } 2^\frac{2n}{n-2m} u_1(x)^\frac{n+2m}{n-2m}\\
& \leq \frac{ (2a)^\frac{n+2m}{4m} }{ 2^\frac{3n+2m}{n-2m} } i_0^\frac{n+2m}{n-2m} 2^\frac{2n}{n-2m} u_1(x)^\frac{n+2m}{n-2m} ~~~~~~ \textmd{by} ~ \eqref{defi0}\\
& \leq k(x)^\frac{n+2m}{4m} \left( \frac{i_0}{2} u_1(x) \right)^\frac{n+2m}{n-2m} ~~~~~~ \textmd{by} ~ \eqref{defab}\\
& \leq k(x)^\frac{n+2m}{4m} \left( \sum\limits_{i=1}^\infty u_i(x) \right)^\frac{n+2m}{n-2m} ~~~~~~ \textmd{by} ~ \eqref{12xx1xx22}.
\endaligned
$$
Therefore, for $k(x) < 1$ (which implies that $|x| > \delta$) and $v \geq 0$, 
$$
U(x) \leq k(x)^\frac{n-2m}{4m} \sum\limits_{i=1}^\infty u_i(x).
$$
From  \eqref{defMZ}, for $k(x) < 1$ and $v \geq 0$,  we have
$$
\aligned
U(x) & \leq \frac{ \left( \sum_{i=1}^\infty u_i(x) - U(x) \right) k(x)^\frac{n-2m}{4m} }{ 1 - k(x)^\frac{n-2m}{4m} }\\
& \leq Z(k(x), P(x,v)).
\endaligned
$$
Thus, by the definition of $F$ we obtain  for   $x \in \R^n$ and $v \geq 0$ that 
$$
\aligned
\overline{H}(x,v) & = f(U(x), k(x), P(x,v))\\
& = k(x) \left ( v + \sum\limits_{i=1}^\infty u_i(x) \right)^\frac{n+2m}{n-2m} - \sum\limits_{i=1}^\infty u_i(x)^\frac{n+2m}{n-2m},
\endaligned
$$
which together with \eqref{LmuiLsui}, \eqref{underH}, \eqref{defu0} and \eqref{uHHoH} implies that 
$$
u : = u_0 + \sum\limits_{i=1}^\infty u_i
$$
is a $C^{2m}$ positive solution of 
\begin{equation}\label{defu}
\begin{cases}
\kappa(x) u^\frac{n+2m}{n-2m} \leq (-\Delta)^m u \leq k(x)  u^\frac{n+2m}{n-2m}\\
(-\Delta)^s u > 0, ~ s = 1,2,\dots,m-1
\end{cases} ~~~~~~ \textmd{in} ~ \R^n \setminus \{0\}.
\end{equation}
It follows from \eqref{defuixi}, \eqref{sumuminu1} and \eqref{defu0} that $u$ satisfies \eqref{uneqphi} and \eqref{ubehainf}.

Define $K:\R^n \to (0,\infty)$ by
\begin{equation}\label{defK}
K(x) = \frac{(-\Delta)^m u(x)}{u(x)^\frac{n+2m}{n-2m}} ~~~~~~ \textmd{for} ~ x \in \R^n \setminus \{0\}
\end{equation}
and $K(0) = 1$. Then
\begin{equation}\label{KxHxu0x}
K(x) = \frac{ H(x,u_0(x)) + \sum_{i=1}^\infty u_i(x)^\frac{n+2m}{n-2m} }{ \left( u_0(x) + \sum_{i=1}^\infty u_i(x) \right)^\frac{n+2m}{n-2m} } ~~~~~~ \textmd{for} ~ x \in \R^n \setminus \{0\}
\end{equation}
and hence $K \in C^1(\R^n \setminus \{0\})$. It follows from \eqref{defu} and \eqref{defK} that
\begin{equation}\label{kKk}
\kappa(x) \leq K(x) \leq k(x) ~~~~~~ \textmd{for} ~ x \in \R^n \setminus \{0\}.
\end{equation}
Recall that $\kappa$, $k \in C^1(\R^n)$ and $\kappa(0) = K(0) = k(0) = 1$, we get $K \in C(\R^n)$,
\begin{equation}\label{nablakKk}
\nabla\kappa(0) = \nabla K(0) = \nabla k(0) = 0
\end{equation}
and
\begin{equation}\label{keqKeqk}
\kappa(x) = K(x) = k(x) ~~~~~~ \textmd{for} ~ |x| \geq 2 \delta_1.
\end{equation}

\vskip0.10in

{\it  Step 5.  Showing   that $K \in C^1(\R^n)$. }  We only need to show that
\begin{equation}\label{limnablaK}
\lim\limits_{|x| \to 0^+} \nabla K(x) = 0.
\end{equation}
Let $S = \{ x \in \R^n \setminus \{0\} ~ : ~ \underline{H}(x,u_0(x)) < H(x,u_0(x)) \}$. It follows from \eqref{underH} and \eqref{KxHxu0x} that
\begin{equation}\label{defS}
S = \{ x \in \R^n \setminus \{0\} ~ : ~ \kappa(x) < K(x) \}. 
\end{equation}
By  \eqref{kKk}, \eqref{nablakKk} and \eqref{defS} we obtain 
\begin{equation}\label{nkaeqnK}
\nabla \kappa(x) = \nabla K(x) ~~~~~~ \textmd{for} ~ x \in \R^n \setminus S
\end{equation}
and thus \eqref{limnablaK} holds for $x \in (\R^n \setminus \{0\}) - S$.  
Next  we show that \eqref{limnablaK} holds for $x\in S$.   It follows from \eqref{underH} and \eqref{defH} that
\begin{equation}\label{SHMUZ}
\begin{cases}
H(x,u_0(x)) = M(\kappa(x),P_0(x))\\
U(x) > Z(\kappa(x),P_0(x))
\end{cases} ~~~~~~ \textmd{for} ~ x \in S, 
\end{equation}
where $P_0(x):= P(x,u_0(x))$. Since $\kappa \geq k_j$ in $B_{2\rho_j}(x_j)$, by \eqref{defMZ}  we have 
\begin{equation}\label{UxZkjMj}
U(x) > Z(k_j,P_0(x)) = M_j^\frac{n-2m}{4m} P_0(x) ~~~~~~ \textmd{for} ~ x \in S \cap B_{2\rho_j}(x_j), ~ j \geq 1.
\end{equation}
For $x' \in (\R^n \setminus \{0\}) - \bigcup\limits_{i=1}^\infty B_{2\rho_i}(x_i)$, $\kappa(x') = k(x')$. Hence, by \eqref{kKk} and \eqref{defS}  we know that $x' \not \in S$. Consequently,
\begin{equation}\label{SsubBi}
S \subset \bigcup\limits_{i=1}^\infty B_{2\rho_i}(x_i).
\end{equation}
For $j \geq 1$ and $x \in S\cap B_{2\rho_j}(x_j)$,  by \eqref{UxZkjMj}  we have 
$$
U(x) > \frac{ k_j^\frac{n-2m}{4m} \left( \sum\limits_{i=1}^\infty u_i(x) - U(x) \right) }{ 1 - k_j^\frac{n-2m}{4m} },
$$
therefore
$$
U(x)> k_j^\frac{n-2m}{4m} \sum\limits_{i=1}^\infty u_i(x).
$$
Hence
\begin{equation}\label{sumijuifuj}
\sum\limits_{i=1, i\neq j}^\infty u_i(x)^\frac{n+2m}{n-2m} \geq f \left( u_j(x), k_j^\frac{n+2m}{4m}, \sum\limits_{i=1, i\neq j}^\infty u_i(x) \right) ~~~~~~ \textmd{for} ~ x \in S \cap B_{2\rho_j}(x_j), ~ j \geq 1.
\end{equation}
However, by Lemma \ref{Taliaferro}, \eqref{minBjww}, \eqref{sumuminu1} and \eqref{3seq2},  we have for $1 \leq j \leq i_0$ and $x \in S \cap B_{2\rho_j}(x_j)$ that 
$$
\aligned
\frac{ \sum_{i=1, i \neq j}^\infty u_i(x)^\frac{n+2m}{n-2m} }{ f \left( 0, k_j^\frac{n+2m}{4m}, \sum_{i=1, i\neq j}^\infty u_i(x) \right)} & = \frac{ \sum_{i=1, i \neq j}^\infty u_i(x)^\frac{n+2m}{n-2m} }{ k_j^\frac{n+2m}{4m} \left( \sum_{i=1, i\neq j}^\infty u_i(x) \right)^\frac{n+2m}{n-2m}}\\
& \leq \frac{ 1 + \left( \frac{1}{3} \right)^{n-2m} }{ k_j^\frac{n+2m}{4m} \left( 1 + \frac{n+2m}{n-2m} \left( \frac{1}{3} \right)^{n-2m} \right) }\\
& < 1.
\endaligned
$$
Thus, by the property of $f$, \eqref{sumijuifuj} and \eqref{minZkjijw0},
\begin{equation}\label{ujZkjsumw0}
u_j(x) > Z \left( k_j^\frac{n+2m}{4m}, \sum\limits_{ i=1, i \neq j }^\infty u_i(x) \right) > w(0) ~~~~~~ \textmd{for} ~ x \in S \cap B_{2\rho_j}(x_j), ~ 1 \leq j \leq i_0, 
\end{equation}
which together with \eqref{wxjljw0Z} implies that
\begin{equation}\label{SB2SBi0}
S \cap B_{2\rho_j}(x_j) = S \cap B_{\rho_j}(x_j) ~~~~~~ \textmd{for} ~ 1 \leq j \leq i_0.
\end{equation}
It follows from \eqref{minZkjsumMj} and \eqref{ujZkjsumw0} that
\begin{equation}\label{ujMjji0}
u_j \geq C M_j^\frac{(1-\alpha)(n-2m)}{4m} ~~~~~~ \textmd{in} ~ S \cap B_{2\rho_j}(x_j) , ~ 1 \leq j \leq i_0.
\end{equation}

For $j>i_0$ and $x \in B_{2 \rho_j}(x_j)$,  by Lemma \ref{Taliaferro}, \eqref{12xx1xx22}, \eqref{sumu12} and \eqref{3seq2} we get
$$
\aligned
\frac{ \sum_{i=1, i\neq j}^\infty u_i(x)^\frac{n+2m}{n-2m} }{ f \left( 0, k_j^\frac{n+2m}{4m}, \sum_{i=1, i\neq j}^\infty u_i(x) \right)} & = \frac{ \sum_{i=1, i\neq j}^\infty u_i(x)^\frac{n+2m}{n-2m} }{ k_j^\frac{n+2m}{4m}\left( \sum_{i=1, i\neq j}^\infty u_i(x) \right)^\frac{n+2m}{n-2m}}\\
& \leq \frac{ 1 + \frac{1}{2} }{ k_j^\frac{n+2m}{4m} \left( 1 + \frac{n+2m}{n-2m} \frac{1}{2} \right) } < 1.
\endaligned
$$
Thus, by the property of $f$, \eqref{sumijuifuj} and \eqref{Mjminu1},
$$
u_j(x) > Z \left( k_j^\frac{n+2m}{4m}, \sum\limits_{i=1, i\neq j}^\infty u_i(x) \right) > Z \left( k_j^\frac{n+2m}{4m}, \frac{1}{ 2 M_j^\frac{\alpha(n-2m)}{4m} } \right) ~~~ \textmd{for} ~ x\in S \cap B_{2\rho_j}(x_j), ~ j > i_0.
$$
Therefore it follows from \eqref{wxjljw0Z} and \eqref{SB2SBi0} that
\begin{equation}\label{SB2SB}
S \cap B_{2\rho_j}(x_j) = S \cap B_{\rho_j}(x_j) ~~~~~~ \textmd{for}~ j \geq 1
\end{equation} 
and it follows from \eqref{ZkjMj} and \eqref{ujMjji0} that
\begin{equation}\label{ujMjj1}
u_j \geq C M_j^\frac{(1-\alpha)(n-2m)}{4m} ~~~~~~ \textmd{in} ~ S \cap B_{2\rho_j}(x_j) , ~ j \geq 1.
\end{equation}

Recall \eqref{KxHxu0x} and \eqref{SHMUZ},  we have 
$$
K(x) = \frac{ M_j P_0(x)^\frac{n+2m}{n-2m} + U(x)^\frac{n+2m}{n-2m} }{ ( P_0(x) + U(x) )^\frac{n+2m}{n-2m} } = \frac{ M_j \left( \frac{P_0(x)}{U(x)} \right)^\frac{n+2m}{n-2m} + 1 }{ \left( \frac{P_0(x)}{U(x)} + 1 \right)^\frac{n+2m}{n-2m} } ~~~~~~ \textmd{for} ~ x \in S \cap B_{\rho_j}(x_j) , ~ j \geq 1.
$$
Thus
$$
\nabla K = \frac{n+2m}{n-2m} \left( \frac{ M_j \left( \frac{P_0}{U} \right)^\frac{4m}{n-2m} - 1 }{ \left( \frac{P_0}{U} + 1 \right)^\frac{2n}{n-2m}} \right) \left( \nabla \frac{P_0}{U} \right) ~~~~~~ \textmd{in} ~ S \cap B_{\rho_j}(x_j), ~ j \geq 1
$$
and hence, by \eqref{UxZkjMj},
\begin{equation}\label{nablaK3t}
\aligned
|\nabla K| & \leq \frac{n+2m}{n-2m} \left| \nabla\frac{P_0}{U} \right|\\
& \leq \frac{n+2m}{n-2m} \left[ \left| \nabla\frac{u_0}{U} \right| + \left| \nabla\frac{\sum_{i=1, i \neq j}^\infty u_i}{U} \right| + \left| \nabla\frac{u_j}{U} \right| \right] ~~~~~~ \textmd{in} ~ S \cap B_{\rho_j}(x_j), ~ j \geq 1.
\endaligned
\end{equation}
Now we estimate these three terms.  By  \eqref{sumnablauiuinnablaui} and \eqref{ujMjj1},  
$$
\aligned
\left| \nabla\frac{1}{U} \right| & = \left| \nabla \left[ (U^\frac{n+2m}{n-2m})^{\frac{2m-n}{2m+n}} \right] \right| = \left| \frac{n-2m}{n+2m} (U^\frac{n+2m}{n-2m})^{-\frac{n-2m}{n+2m}-1}\nabla U^\frac{n+2m}{n-2m} \right|\\
& = \left| \frac{ \sum_{i=1, i \neq j}^\infty u_i^\frac{4m}{n-2m}\nabla u_i + u_j^\frac{4m}{n-2m}\nabla u_j }{ U^\frac{2n}{n-2m} } \right|\\
& \leq C \left( \frac{ M_j^\frac{1}{2m} }{ M_j^\frac{(1-\alpha)n}{2m}} + \frac{|\nabla u_j|}{u_j^2} \right) ~~~~~~ \textmd{in} ~ S \cap B_{\rho_j}(x_j), ~ j \geq 1.
\endaligned
$$
Hence, by \eqref{sumuiuinC}, \eqref{sumnablauiuinnablaui}, \eqref{defu0} and \eqref{ujMjj1},
\begin{equation}\label{nablau0U}
\aligned
\left| \nabla \frac{u_0}{U} \right| & = \left| \frac{\nabla u_0}{U} + u_0 \nabla\frac{1}{U} \right|\\
& \leq C \left( \frac{ |\nabla u_0| }{ M_j^\frac{(1-\alpha)(n-2m)}{4m} } +  \frac{ M_j^\frac{1}{2m} }{ M_j^\frac{(1-\alpha)n}{2m} } + \frac{|\nabla u_j|}{u_j^2} \right) ~~~~~~ \textmd{in} ~ S \cap B_{\rho_j}(x_j), ~ j \geq 1
\endaligned
\end{equation}
and
\begin{equation}\label{nablasumU}
\aligned
\left| \nabla\frac{\sum_{i=1, i \neq j}^\infty u_i}{U} \right| & \leq \left| \nabla\frac{1}{U} \right| \sum_{i=1, i \neq j}^\infty u_i +  \frac{1}{U} \left| \nabla\sum_{i=1, i \neq j}^\infty u_i \right|\\
& \leq C \left( \frac{ M_j^\frac{1}{2m} }{ M_j^\frac{(1-\alpha)n}{2m} } + \frac{|\nabla u_j|}{u_j^2} + \frac{ M_j^\frac{1}{2m} }{ M_j^\frac{(1-\alpha)(n-2m)}{4m} } \right)\\
& \leq C \left( \frac{|\nabla u_j|}{u_j^2} + \frac{ M_j^\frac{1}{2m} }{ M_j^\frac{(1-\alpha)(n-2m)}{4m} } \right) ~~~~~~ \textmd{in} ~ S \cap B_{\rho_j}(x_j), ~ j \geq 1.
\endaligned
\end{equation}
Since
$$
\aligned
\nabla \frac{u_j}{U} = & \nabla \left( \frac{\sum_{i=1}^\infty u_i^\frac{n+2m}{n-2m}}{u_j^\frac{n+2m}{n-2m}} \right)^{-\frac{n-2m}{n+2m}} = \nabla \left( 1 + \frac{\sum_{i=1, i\neq j}^\infty u_i^\frac{n+2m}{n-2m}}{u_j^\frac{n+2m}{n-2m}} \right)^{-\frac{n-2m}{n+2m}}\\
= & -\frac{n-2m}{n+2m} \left( 1 + \frac{\sum_{i=1, i \neq j}^\infty u_i^\frac{n+2m}{n-2m}}{u_j^\frac{n+2m}{n-2m}} \right)^{-\frac{2n}{n+2m}} \Bigg[ \frac{ \nabla\sum_{i=1, i \neq j}^\infty u_i^\frac{n+2m}{n-2m} }{u_j^\frac{n+2m}{n-2m}}\\
& -\frac{n+2m}{n-2m} \left( \frac{\nabla u_j}{u_j^\frac{2n}{n-2m}} \right) \sum\limits_{i=1, i \neq j}^\infty u_i^\frac{n+2m}{n-2m} \Bigg],
\endaligned
$$
and by \eqref{sumuiuinC}, \eqref{sumnablauiuinnablaui}, \eqref{defu0} and \eqref{ujMjj1}, we get
\begin{equation}\label{nablaujU}
\left| \nabla \frac{u_j}{U} \right| \leq C \left( \frac{M_j^\frac{1}{2m}}{M_j^\frac{(1-\alpha)(n+2m)}{4m}} + \frac{|\nabla u_j|}{u_j^\frac{2n}{n-2m}} \right) ~~~~~~ \textmd{in} ~ S \cap B_{\rho_j}(x_j), ~ j \geq 1.
\end{equation}
By \eqref{nablaK3t}-\eqref{nablaujU}, we obtain for $j \geq 1$ and  $x \in S \cap B_{\rho_j}(x_j)$  that 
\begin{equation}\label{balabala} 
|\nabla K| \leq  C \left( \frac{ |\nabla u_0| }{ M_j^\frac{(1-\alpha)(n-2m)}{4m} } + \frac{|\nabla u_j|}{u_j^2} + \frac{ M_j^\frac{1}{2m} }{ M_j^\frac{(1-\alpha)(n-2m)}{4m} } + \frac{|\nabla u_j|}{u_j^\frac{2n}{n-2m}}\right).
\end{equation}

We now estimate $\nabla u_0$ in $B_{\rho_j}(x_j)$. By \eqref{defu0}  there exists a continuous function $h:\overline{B}_2 \to \R$ which satisfies $(-\Delta)^m h=0$  in $B_2$ such that  
$$
u_0(x) = \gamma_{n,m} \int_{B_4}\frac{H(y,u_0(y))}{|x-y|^{n-2m}}dy + h(x) ~~~~~~ \textmd{for} ~ 0 < |x| \leq 2, 
$$
where $\gamma_{n,m} = \Gamma(n/2-m)/(2^{2m}\pi^{n/2}\Gamma(m))$.
By \eqref{HxvLmvx}, \eqref{HxvLmvx2} and \eqref{defu0}, 
$$
H(x,u_0(x)) \leq 
\begin{cases}
(2w(0))^\frac{n+2m}{n-2m} M_j ~~~~~~  & \textmd{in} ~ B_{\rho_j}(x_j), ~ j \geq 1, \\
(2w(0))^\frac{n+2m}{n-2m} b & \textmd{in} ~ (\R^n \setminus \{0\}) - \bigcup\limits_{i=1}^\infty B_{\rho_i}(x_i).
\end{cases}
$$
It follows from \eqref{defu0} that $|h(x)| < C$ for $|x| \leq 2$. Thus  $|\nabla h(x)| < C$ for $|x| \leq 1$. Hence, for $x \in B_{\rho_j}(x_j)$,
$$
\aligned
|\nabla u_0(x)|\leq C \int_{B_4}\frac{H(y,u_0(y))}{|x-y|^{n-2m+1}}dy + C 
\leq C [ I_1(x) + I_2(x) + I_3(x) ] + C, 
\endaligned
$$
where
$$
I_1(x) := \int_{B_{\rho_j}(x_j)} \frac{M_j}{|x-y|^{n-2m+1}} dy \leq C M_j \rho_j^{2m-1} \leq C M_j^\frac{1}{2m} ~~~~~~ \textmd{for} ~ x \in B_{\rho_j}(x_j)
$$
by Lemma \ref{sim}, and
$$
\aligned
I_2(x) & := \sum\limits_{i=1, i \neq j}^\infty \int_{B_{\rho_i}(x_i)} \frac{M_i}{|x-y|^{n-2m+1}} dy \leq C \sum\limits_{i=1, i \neq j}^\infty \frac{M_i \rho_i^n}{ \textmd{dist}( B_{\rho_i}(x_i), B_{\rho_j}(x_j) )^{ n-2m+1 } } \\
& \leq C \sum\limits_{i=1, i \neq j}^\infty \frac{\rho_i^{n-2m}}{ 2^i (\rho_i +\rho_j)^{ n-2m+1 } } \leq \frac{C}{\rho_j} \sim C 2^\frac{j}{2m} M_j^\frac{1}{2m} \leq C M_j^\frac{1+\alpha}{2m} ~~~~~~ \textmd{for} ~ x \in B_{\rho_j}(x_j)
\endaligned
$$
by Lemma \ref{sim}, \eqref{distBiBj} and \eqref{3seq2},  and 
$$
I_3(x) := \int_{ B_4 - \bigcup\limits_{i=1}^\infty B_{\rho_i}(x_i) } \frac{1}{|x-y|^{n-2m+1}} dy \leq C  ~~~~~~ \textmd{for} ~ x \in B_{\rho_j}(x_j).
$$
Thus
\begin{equation}\label{nablau0Mj}
|\nabla u_0| \leq C M_j^\frac{1 + \alpha}{2m}  ~~~~~~ \textmd{in} ~ B_{\rho_j}(x_j), ~ j \geq 1.
\end{equation}
Since $n \geq 2m + 4$, it follows from \eqref{nablau0Mj} that
\begin{equation}\label{estu0Mj}
\frac{ |\nabla u_0| }{ M_j^\frac{(1-\alpha)(n-2m)}{4m} } \leq \frac{ C M_j^\frac{1 + \alpha}{2m} }{ M_j^\frac{1-\alpha}{m} } \leq \frac{C}{ M_j^\frac{1-3\alpha}{2m} }.
\end{equation}

Finally, we estimate $|\nabla u_j|/u_j^2$ and $|\nabla u_j|/u_j^\frac{2n}{n-2m}$ in $S \cap B_{\rho_j}(x_j)$. Let
$$
s_j = \inf \{ s>0 ~ : ~  S \cap B_{\rho_j}(x_j) \subset B_s(x_j)\}
$$
and $\tilde{u}_j(s) = \psi(s,\lambda_j)$. Then $s_j \leq \rho_j$ and $\tilde{u}_j(s) = u_j(x)$ when $|x-x_j| = s$. By \eqref{ujMjj1}  we have
$$
\tilde{u}_j(s) \geq C M_j^\frac{(1-\alpha)(n-2m)}{4m} ~~~~~~ \textmd{for} ~ 0 \leq s \leq s_j, ~ j \geq 1.   
$$
It follows from Lemma \ref{sim} that
$$
\left( \frac{\lambda_j}{\lambda_j^2+s_j^2} \right)^{2m} \geq C M_j^{1-\alpha} \geq C \left( \frac{\epsilon_j^\frac{2m}{n-2m}}{2^j\lambda_j^m} \right)^{1-\alpha}
$$
and hence,   by \eqref{3seq2},
$$
\aligned
s_j & \leq C\left( \frac{2^j}{\epsilon_j^\frac{2m}{n-2m}} \right)^\frac{1-\alpha}{4m} \lambda_j^\frac{3-\alpha}{4} \leq C (\lambda_j^{-\alpha})^\frac{1-\alpha}{4m} \lambda_j^\frac{3-\alpha}{4}\\
& \leq C (\lambda_j^{-\alpha})^\frac{1}{4m} \lambda_j^\frac{3-\alpha}{4} \leq C \lambda_j^\frac{3m-\alpha(m+1)}{4m} ~~~~~~ \textmd{for} ~ j \geq 1.
\endaligned
$$
Since $n \geq 2m+4$, we have for $0 \leq s \leq s_j$ and $j \geq 1$ that 
\begin{equation}\label{esttildeuj1}
\aligned
\frac{-\tilde{u}_j'(s)}{\tilde{u}_j(s)^2} & = \frac{n-2m}{c_{n,m}} \frac{(\lambda_j^2+s^2)^\frac{n-2m-2}{2}s}{\lambda_j^\frac{n-2m}{2}}\\
& \leq \frac{n-2m}{c_{n,m}} \frac{(\lambda_j^2+s_j^2)^\frac{n-2m-2}{2}s_j}{\lambda_j^\frac{n-2m}{2}}\\
& \leq C \frac{\left( \lambda_j^2+\lambda_j^\frac{3m-\alpha(m+1)}{2m} \right)^\frac{n-2m-2}{2}\lambda_j^\frac{3m-\alpha(m+1)}{4m}}{\lambda_j^\frac{n-2m}{2}}\\
& \leq C \lambda_j^\frac{m(n-2m-3)-\alpha(m+1)(n-2m-1)}{4m}
\endaligned
\end{equation}
and
\begin{equation}\label{esttildeuj2}
\aligned
\frac{-\tilde{u}_j'(s)}{\tilde{u}_j(s)^\frac{2n}{n-2m}} & = \frac{n-2m}{c_{n,m}^\frac{n+2m}{n-2m}}\frac{(\lambda_j^2+s^2)^\frac{n+2m-2}{2}s}{\lambda_j^\frac{n+2m}{2}}\\
&\leq \frac{n-2m}{c_{n,m}^\frac{n+2m}{n-2m}}\frac{(\lambda_j^2+s_j^2)^\frac{n+2m-2}{2}s_j}{\lambda_j^\frac{n+2m}{2}}\\
&\leq C \frac{\left(\lambda_j^2+\lambda_j^\frac{3m-\alpha(m+1)}{2m}\right)^\frac{n+2m-2}{2}\lambda_j^\frac{3m-\alpha(m+1)}{4m}}{\lambda_j^\frac{n+2m}{2}}\\
&\leq C \lambda_j^\frac{m(n+2m-3)-\alpha(m+1)(n+2m-1)}{4m}.
\endaligned
\end{equation}
We pick $\alpha = m/(6m+6)$.  Since $n \geq 2m +4$,  by \eqref{balabala}-\eqref{esttildeuj2} we get 
\begin{equation}\label{nablaKMl}
|\nabla K| \leq C \left( \frac{1}{M_j^\frac{m+2}{4m(m+1)}} + \lambda_j^\frac{1}{8} \right) ~~~~~~ \textmd{in} ~ S \cap B_{\rho_j}(x_j), ~ j \geq 1.
\end{equation}
Hence, it follows from Lemma \ref{sim}, \eqref{3seq2}, \eqref{SsubBi} and \eqref{SB2SB} that \eqref{limnablaK} also holds for $x\in S$.  Thus we have  $K \in C^1(\R^n)$.

By sufficiently increasing $k_i$ for each $i \geq 1$, we can force $\kappa$ to satisfy
\begin{equation}\label{kkappaC1}
\| k-\kappa \|_{C^1(\R^n)} < \frac{\varepsilon}{4}
\end{equation}
by \eqref{1kirhoi}, \eqref{defkappa}, \eqref{nablakappa}, \eqref{KxHxu0x} and \eqref{keqKeqk}. We also can force $K$ to satisfy
$$
|\nabla (K-\kappa)| = |\nabla (K-(\kappa-k))| \leq |\nabla K| + \frac{\varepsilon }{4} \leq \frac{\varepsilon}{2} ~~~~~~ \textmd{in} ~ S
$$
by \eqref{SsubBi}, \eqref{SB2SB} and \eqref{nablaKMl}. Thus by \eqref{nkaeqnK}, $|\nabla (K-\kappa)| \leq \varepsilon /2$ in $\R^n$. It follows from \eqref{kKk} and \eqref{kkappaC1} that $K$ satisfies \eqref{KkC1}. The proof of Theorem \ref{thmC1K}  is completed.   
\hfill$\square$

\vskip0.40in

\noindent{X. Du and H. Yang}\\
Department of Mathematics, The Hong Kong University of Science and Technology\\
Clear Water Bay, Kowloon, Hong Kong\\
E-mail addresses:  xduah@connect.ust.hk (X. Du)~~~~~  mahuiyang@ust.hk (H. Yang)

\end{document}